%% file: Abc_paper_2.tex
[1994/06/01]

\documentclass[reqno,intlimits,english,centertags]{elsarticle}

\usepackage{longtable,geometry,amsmath}
\usepackage[T1]{fontenc}
\usepackage{lmodern}

\usepackage{ae}

\usepackage{multirow}
\usepackage{epsfig}
\usepackage{calc}
\usepackage{amssymb}
\usepackage{multicol}

\usepackage{calc}
\usepackage{babel}
\usepackage{xspace}
\usepackage{ifthen}
\makeatletter
  \newtheorem{theorem}{Theorem}

\newtheorem{proposition}[theorem]{Proposition}

\newproof{proof}{Proof}
\newdefinition{remark}[theorem]{Remark}
\newdefinition{definition}[theorem]{Definition}
\newcommand{\textupmd}[1]{\textup{\textmd{#1}}}

\newcommand{\SpDim}{N}
\newcommand{\numberspacefont}{\boldsymbol}
\newcommand{\R}{\numberspacefont{R}}

\newcommand{\N}{\numberspacefont{N}}

\newcommand{\Oset}{\varOmega}

\newcommand{\const}{\gamma}

\newcommand{\maxT}{T}

\newcommand{\@di}{\textupmd{d}}
\newcommand{\di}{\,\@di}
\newcommand{\grad}{\operatorname{\nabla}}
\newcounter{da@aux}
\newcounter{da@auxb}
\newcommand{\onlyifnotone}[1]{\ifnum\arabic{#1}=1\else\arabic{#1}\fi}
\newcommand{\der}[3][1]{
\ifthenelse{ \equal{#1}{1} }{ \def\@deraux{} }{ \def\@deraux{#1} }
\frac{\@di^{\@deraux}#2}
            {\@di #3^{\@deraux}}}
\newcommand{\pder}[2]{\setcounter{da@aux}{-1}\setcounter{da@auxb}{0}%
\def\Last@Tok{\relax}%
\def\List@Tok{\relax}%
\Check@Tok#2\relax%
\frac{\partial^{\onlyifnotone{da@auxb}} #1}{\List@Tok}}
\newcommand{\Check@Tok}[1]{\xdef\@Argom{#1}
\ifnum\arabic{da@aux}=-1 \setcounter{da@aux}{0}
\xdef\Last@Tok{\@Argom}\fi%
\ifx#1\relax \let\next=\relax%
\xdef\ProvList@Tok{\List@Tok}%
\xdef\List@Tok{\ProvList@Tok \partial\Last@Tok^{\onlyifnotone{da@aux}}}%
\addtocounter{da@auxb}{\value{da@aux}}%
\else%
\ifx\Last@Tok\@Argom\stepcounter{da@aux}%
\else
\xdef\ProvList@Tok{\List@Tok}%
\xdef\List@Tok{\ProvList@Tok \partial\Last@Tok^{\onlyifnotone{da@aux}}}%
\xdef\Last@Tok{\@Argom}%
\addtocounter{da@auxb}{\value{da@aux}}
\setcounter{da@aux}{1}%
\fi
\let\next=\Check@Tok
\fi
\next
}

\DeclareMathOperator{\Lapl}{\Delta}

\newcommand{\Lspbasename}{L}
\newcommand{\Lsp}[3][\relax]{\ifthenelse{\equal{#3}{\empty}}%
{{{\Lspbasename}^{#2}}}{{{\Lspbasename}^{#2}}#1(#3#1)}}
\newcommand{\Cspbasename}{C}
\newcommand{\Csp}[3][\relax]{\ifthenelse{\equal{#3}{\empty}}%
{{{\Cspbasename}^{#2}}}{{{\Cspbasename}^{#2}}#1(#3#1)}}

\newcommand{\Diffu}{D}
\newcommand{\unk}{u}

\newcommand{\defeq}{:=}

\newcommand{\Norm}[1]{\left\lVert#1\right\rVert}

\newcommand{\Norma}[2]{\Norm{#1}_{#2}}
\newcommand{\Measbase}[2][*]{\ifthenelse{\equal{#1}{*}}%
{\lvert#2\rvert}{\left|#2\right|}}
\newcommand{\meas}[1]{\Measbase[*]{#1}}

\newcommand{\oo}{\infty}

\newcommand{\eps}{\varepsilon}
\newcommand{\phj}{\varphi}

\newcommand{\textem}[1]{\emph{#1}}

\newcommand{\ignore}[1]{\relax}

\usepackage{pstricks,graphicx}

\newcommand{\tfix}{\bar{t}}

\newcommand{\openpores}{\Mpores_{\openset}}

\newcommand{\inttrace}[1]{#1_{\textup{int}}}
\newcommand{\exttrace}[1]{#1_{\textup{ext}}}
\newcommand{\solspace}{V}

\newcommand{\constcp}{\const_{0}}

\newcommand{\maxconst}{\Lambda}

\newcommand{\ut}{\unk^{\Mper}}

\newcommand{\constidK}{l_{1K}}

\newcommand{\Mper}{\tau}
\newcommand{\Mwid}{\eps}
\newcommand{\Mdpt}{\delta}

\newcommand{\Mcell}{B_{\Mwid}}
\newcommand{\Mcelltwo}{B_{2\Mwid}}

\newcommand{\Mpwid}{\sigma_{\Mwid}}

\newcommand{\Mopen}{\sigma_{\Mper}}

\newcommand{\Mpore}{P^{\Mwid}_i}
\newcommand{\basepore}{P_0}

\newcommand{\Mpores}{\mathcal{P}}
\newcommand{\inter}{\Gamma}
\newcommand{\openset}{A_{\Mper}}

\newcommand{\Mpoints}{z_i^{\Mwid}}
\newcommand{\Mpointsj}{z_m^{\Mwid}}
\newcommand{\mw}{m_{\Mwid}}
\newcommand{\mt}{m_{\Mper}}
\newcommand{\Osett}{\Oset^{\Mper}}

\newcommand{\pwido}{\varLambda_0}

\newcommand{\bdrp}[1]{\partial_{+}#1}

\newcommand{\xN}{x_{\SpDim}}

\newcommand{\Lflxs}{\Phi}

\newcommand{\fstconstNa}{l_{f_{Na}}}
\newcommand{\smlconstNa}{l_{s_{Na}}}
\newcommand{\fstconstK}{l_{f_K}}
\newcommand{\smlconstK}{l_{s_K}}

\newcommand{\mesfun}{M}

\newcommand{\mesfuno}{\mesfun_0}

\newcommand{\limconst}{\varrho}

\newcommand{\fstlNa}{l^{\Mper}_{f_{Na}}}
\newcommand{\smllNa}{l^{\Mper}_{s_{Na}}}
\newcommand{\fstlK}{l^{\Mper}_{f_K}}
\newcommand{\smllK}{l^{\Mper}_{s_K}}

\newcommand{\smlD}{D_1}
\newcommand{\bigD}{D_0}
\newcommand{\constD}{D_0}
\newcommand{\Mbor}{a}
\newcommand{\maxfun}{v}
\newcommand{\outnormal}{\widehat{\nu}}
\makeatother

\newcommand{\bb}[1]{{\mathbb{#1}}}

\begin{document}

\title{A model for enhanced and selective
transport through biological membranes with alternating pores}
\author[sbai]{Daniele Andreucci}
\ead{daniele.andreucci@sbai.uniroma1.it}
\author[sbai]{Dario Bellaveglia}
\ead{dario.bellaveglia@sbai.uniroma1.it}
\author[sbai]{Emilio Nicola Maria Cirillo}
\ead{emilio.cirillo@uniroma1.it}

\address[sbai]{Dept. of Basic and Applied Sciences for Engineering\\
via A.Scarpa 16 00161 Roma Italy}

\begin{abstract}
  We investigate the outflux of ions through the channels in a cell
  membrane. The channels undergo an open/close cycle according to a
  periodic schedule. Our study is based both on theoretical
  considerations relying on homogenization theory, and on Monte Carlo
  numerical simulations. We examine the onset of a limiting boundary
  behavior characterized by a constant ratio between the outflux and
  the local density, in the thermodynamics limit. The focus here is on
  the issue of selectivity, that is on the different behavior of the
  ion currents through the channel in the cases of the selected and
  non-selected species.
\end{abstract}

\maketitle
\input{Ndim/intro}

\input{Ndim/main_results}
\input{onedim/onedim}

\input{onedim/onedim_lattice_model}

\input{onedim/onedim_mc_results}

\input{AB_refs.bbl}

\end{document}

%% file: Ndim/intro.tex
\section{Introduction}
\label{s:intro}

\subsection{The model}
Potassium currents across cell membranes have been widely studied,
since they play many important and different functional roles (see,
e.g., the
reviews \cite{Hille:2001,RCM}).  Indeed, ionic
channels selecting the flux of Potassium ions are ubiquitous in all
organisms.

We confine ourselves here to recall that ionic channels form selective
pores in the cell membrane which open and close and, when in the open
state, allow permeation of ions favoring selection of a species
(Potassium in K$^+$--channels).  The processes turning on and off ion
conduction, i.e., \textit{gating}, and the channel ability to allow
the flux of a particular ionic species, i.e., \textit{selectivity},
are not yet completely understood, but it has been known for some time
in the literature the conjecture that they are functionally linked
\cite{VanDongen:2004}. The idea is that the density of selected ions is
higher in a region close to the pore where the affinity is higher;
when the pore opens, passive diffusion together with such an unbalance
in concentration is enough to cause a selective outflux. 

We model this phenomenon as a diffusion problem in a domain with
alternating pores on the boundary. The latter are holes periodically
and simultaneously cycling through open and closed phases. We model
the affinity to the selected species by setting the corresponding
diffusivity smaller in a suitably chosen small region in the
neighborood of the pores.  Assuming that the pores are many and with
small diameter with respect to the dimensions of the domain and also
that the period of the cycle is much smaller than the characteristic
time of diffusion, it is possible to approximate the problem with its
homogenized version, where the number of pores and the number of
cycles both diverge to infinity and the diffusion in the affinity
region goes to zero.  In \cite{Andreucci:2012} we introduced four
parameters: the distance between neighboring pores $\Mwid$, the
diameter of each pore $\Mpwid$, the period of the opening/closing
cycle $\Mper$, and finally $\Mopen$ as the length of the time
sub-interval of each cycle in which the pores are open. If the
parameters $\Mwid$, $\Mpwid$, $\Mper$ and $\Mopen$ are properly
related, in the asymptotic limit when they become vanishingly small we
obtain a limiting boundary condition of the type
\begin{equation}
\label{eq:intro_lim_cond}
  \bigD\pder{\unk}{\nu}(x,t)
  =
  \varphi(x)
  \bigD
  \unk(x,t)
  \,,
\end{equation}
where $\bigD$ is the diffusivity of the ions in the cytosol, $\varphi$
is a function connected to the distribution and to the shape of pores
and $u$ denotes the concentration of ions. Notice that $\bigD$ appears
on both sides of \eqref{eq:intro_lim_cond} so that the left hand side
equals the flux at the boundary.

However in \cite{Andreucci:2012} we considered the case where no
selection is present, and therefore diffusivity is a constant; this is
for example the case when looking at the flux of Sodium ions through
Potassium channels. When considering flux of $K^{+}$ ions through the
same channels, a variable diffusivity is to be taken into
account. Thus we model the phenomenon by using the Fokker-Planck
Equation $\unk_t- \Lapl(\Diffu \unk)=0$, where we let $\Diffu=\smlD$
in a suitably chosen neighborood of the pores, and $\smlD$ is a
vanishingly small value, in the asymptotic or homogenization limit
described above.  We obtain in the limit an asymptotic boundary
condition of the type \eqref{eq:intro_lim_cond}, but in this case the
function $\varphi$ also depends on the asymptotics of $\smlD$. We will
show that two different asymptotic standards are admissible, in order
to obtain the limiting interface condition
\eqref{eq:intro_lim_cond}. They are discriminated by the limiting
behavior of the ratio ${\Mpwid}/{\sqrt{\smlD\Mopen}}$ according to
the two cases
\begin{alignat}2
 \label{eq:main_3d_fast}
\lim_{\Mwid,\Mper \to 0}
\frac{\Mpwid}{\sqrt{\smlD\Mopen}}
&=
+\oo
\,,
&\qquad&
\text{which we call the case of  \textem{fast pores},}
\\
\label{eq:main_3d_small}
\lim_{\Mwid,\Mper \to 0}
\frac{\Mpwid}{\sqrt{\smlD\Mopen}}
&\le
\pwido
\,,
&\qquad&
\text{which we call the case of \textem{small pores};}
\end{alignat}
here $\pwido$ is a positive constant.  We will see that the
introduction of selection, that is of the asymptotically vanishing
diffusivity $\smlD$, yields a flux enhancement only in the case of
fast pores. The latter case was introduced in \cite{Andreucci:2012}
and tested numerically in \cite{Andreucci:2013b}, and is specific to
evolutive problems, while the small pores behavior is connected with
the stationary case considered for example in
\cite{Sanchez:1982,Murat:1985,Friedman:1995,Cioranescu:2008,Peter:2007,Ansini:2004},
but it does appear in evolutive problems as well. A relevant
difference between the two cases from the point of view of selectivity
will be outlined in Subsection~\ref{ss:discussion}.

Section~\ref{s:mathmod} is devoted to the theoretical analysis of the
diffusion problem. Subsection~\ref{s:N=1} prepares the way to Sections
\ref{s:rw} and \ref{s:discu} where the problem will be attacked via a
stochastic discrete space model.  This approach is strictly related to
that adopted in \cite{Andreucci:2011,Andreucci:2013a,Andreucci:2013b},
in connection with ion currents, and in
\cite{cirillomuntean:2012,cirillomuntean:2013}, in connection with
crowd dynamics.  We note, however, that in the model proposed and
studied in \cite{Andreucci:2011,Andreucci:2013a} the open/close cycles
of the pore are not prescribed a priori on a deterministic schedule.
There gating is realized in a stochastic fashion as a result of a
stochastic flipping of the pore between a low and a high affinity
state.

%% file: Ndim/main_results.tex
\section{Statement of the problem and main results}
\label{s:mathmod}

\subsection{Geometry and alternating pores}
\label{ss:mathmod_geom}
The quantities $\Mwid$, $\Mpwid$, $\Mopen$ $\smlD$ are 
defined as functions of $\Mper$, vanishing as $\Mper\to 0$
\begin{equation*}
  \label{eq:par_tau_dep_ii}
  \Mwid
  =
  \Mwid(\Mper)
  \to
  0
  \,,
  \qquad
  \Mpwid
  =
  \Mpwid(\Mper)
  \to
  0
  \,,
  \qquad
  \Mopen
  =
  \Mopen(\Mper)
  \to
  0
  \,,
  \qquad
  \smlD
  =
  \smlD(\Mper)
  \to
  0
  \,.
\end{equation*}
For reasons of technical simplicity we choose
\begin{equation}
 \label{eq:mathmod_set_hyp_iii}
\Oset
=
(-\Mbor, 
\Mbor)^{\SpDim-1}
\times
(0,\Mbor)
\,,
\qquad
\inter
=
[
-\Mbor
, 
\Mbor
]^{\SpDim-1}
\times
\{
\Mbor
\}
\,,
\end{equation}
for a given $\Mbor>0$. We model the pores as a subset $\Mpores$ of
$\inter$, i.e.,
\begin{equation}
  \label{eq:mathmod_pores_hyp_i}
  \Mpores
  =
  \bigcup_{i=1}^{\mw}
  \Mpore
  \,,
  \qquad
  \Mpore:=\Mpwid\basepore+\Mpoints
  \,,
\end{equation}
where $\mw$ is the number of pores, $\Mpwid>0$, the $\Mpoints$ are
points of $\inter$, $\basepore$ is an open set in $\R^{\SpDim-1}$ such
that $\partial \basepore \in C^3$ and, for a given
$\delta_{0}\in(0,1)$
\begin{equation}
\label{eq:mathmod_pores_hyp_ii}
  (-\delta_0,\delta_0)^{\SpDim-1}
  \subset
  \basepore
  \subset
  (-1,1)^{\SpDim-1}
\,.
\end{equation}
We define the total open phase and the union of the open pores as
\begin{equation}
  \label{eq:mathmod_open_hyp_i}
  \openset
  =
  \bigcup_{j=0}^{\mt-1}
  [j\Mper, j\Mper + \Mopen)
  \,,
  \qquad
  \openpores
  =
  \bigcup_{i=1}^{\mw}
  \Mpore
  \times
  \openset
  \,.
\end{equation}
Here $\Mopen>0$ is the opening interval in each cycle, $\Mper>0$
is the period of the cycle and $\mt$ is the total number of the
cycles, related by
\begin{equation}
  \label{eq:wdef_ass_iii}
  \Mper
  =
  \frac{\maxT}{\mt}
  \,,
  \qquad
  \mt
  \in 
  \N
  \,,
\end{equation}
so that $\mt\to\oo$ as $\Mper\to 0$.
The lenght $\Mwid$ satisfies the 
following requirements. Define
the domains $\Mcell(\Mpoints):=\Mcell(0)+z_i^{\Mwid}$, where
\begin{equation*}
  \Mcell(0)
  =
  (-\Mwid,\Mwid)^{\SpDim-1}
\times
(-\Mwid,0)
\,.
\end{equation*}
Then we assume that
\begin{equation}
 \label{eq:wdef_ass_i_bis}
\Mcell(\Mpoints)
\subset 
\Oset
\,;
\qquad
\Mcelltwo(\Mpoints)
\cap
 \Mcelltwo(\Mpointsj)
= 
\emptyset
\,,
\quad
\text{for any $i\neq m$}
\,.
\end{equation}
In addition, we stipulate the existence of a function 
$\mesfun(x)\in L^{\oo}(\inter)$ such that 
$\mesfun(x)\ge 0$, $\mesfun(x)\not\equiv 0$ and
\begin{equation}
 \label{eq:wdef_ass_ii}
\lim_{\Mper\to 0}
\sum_{i=1}^{\mw}
\sum_{j=0}^{\mt-1}
\Mper
\Mwid^{\SpDim-1}
\phj(\Mpoints,j\Mper)
=
\int_0^{\maxT}
\int_{\inter}
\mesfun(x)
\phj(x,t)
\di s
\di t
\,,
\end{equation}
for any $\phj\in C(\inter\times [0,\maxT])$.
Then taking $\phj \equiv 1$ in \eqref{eq:wdef_ass_ii} we get
\begin{equation}
 \label{eq:wdef_ass_v}
\mw
\sim
\frac{\mesfuno}{\Mwid^{\SpDim-1}}
\,,
\qquad
\text{as $\Mwid \to 0$}
\,,
\qquad
\text{where 
$\mesfuno=\int_{\inter} \mesfun(x) \di s>0 $.} 
\end{equation}
The function $\mesfun$ measures the
density of pores on $\inter$.

Then the diffusivity for Potassium ions is given by
\begin{equation}
 \label{eq:N_D_definition}
 \Diffu_K(x)
 =
 \begin{cases}
  \smlD
  \,,
  \quad&
  x\in \Osett
  \\
  \bigD
  \,,
  \quad&
  x\in \Oset\setminus \Osett
 \end{cases}
\,,
\quad
\text{where}
\quad
 \Osett=\bigcup_{i=1}^{\mw}
 \Mcell(\Mpoints)
 \,,
\end{equation}
$0<\smlD(\Mper)<\bigD$ and $\bigD$ is a constant. We understand that
$\smlD(\Mper)\to0$ as $\Mper\to0$ unless otherwise noted (as in
Subsection~\ref{s:N=1}). In turn the diffusivity for Sodium is constant
and, given the theoretical character of our analysis, denoted for
simplicity of comparison with the same symbol $\bigD$:
\begin{equation}
 \label{eq:N_D_definition_Na}
 \Diffu_{Na}(x)
 =
  \constD
  \,,
  \quad
  x\in \Oset
  \,.
\end{equation}

\subsection{Formulation of the approximating problem}
\label{ss:mathmod_form}
For any set $A\subset \R^{\SpDim}$ we use below the notation $\bdrp{A}
= \partial{A}\cap\{\xN<\Mbor\}$.  For any function $F(x,t)$, with
$x\in \Oset$, we will denote with $[F]$ the jump across $\bdrp\Osett$
\begin{equation*}
  [F]
  \defeq
  \exttrace{F}
  -
  \inttrace{F}
  \,,
  \qquad
  \inttrace{F}
  \defeq
  \text{trace of $F|_{\Osett}$ on $\bdrp\Osett$,}
  \quad 
  \exttrace{F}
  \defeq
  \text{trace of $F|_{\Oset\setminus\Osett}$ on $\bdrp\Osett$.}
\end{equation*}
We consider the problem for the concentration $\ut\ge 0$
\begin{alignat}2
\label{eq:N_prb1} 
\ut_{t}
- 
\Lapl(\Diffu\ut)
&=
0
\,,
 \qquad  
 (x,t)\in \Oset\times[0,\maxT]
 \,,
 \\
\label{eq:N_open_pore}
 \ut 
 &=
 0
  \,, 
  \qquad 
  (x,t)
  \in
  \openpores
  \,,
  \\
\label{eq:N_closed_pore} 
\grad(\Diffu\ut)
\cdot
\outnormal
&=0 
\,, 
\qquad 
(x,t)
\in
\{\inter\times[0,\maxT]\}
\setminus
\openpores
\,,
\\
\label{eq:N_ext_bc} 
\grad(\Diffu\ut)
\cdot
\outnormal
&=
0\,,  
\qquad 
(x,t)
\in
\bdrp
\Oset
\times
[0,\maxT]
\\
\label{eq:N_jump}
[\Diffu \ut]=[\grad(\Diffu \ut)\cdot \nu]
&=
0
\,,
\qquad
(x,t)\in \bdrp\Osett \times[0,\maxT]
\,,
\\
\label{eq:N_initial} 
\ut(x,0)
&=
\unk_0(x)
\,, 
\qquad 
x\in \Oset
\,.
\end{alignat}
Here $\outnormal$ is the outer normal to $\Oset$, $\nu$ is the outer
normal to $\Osett$ and we suppose that $\unk_0 \in L^{\oo}(\Oset)$.
In the case of Sodium the diffusivity is constant in the whole domain
so that \eqref{eq:N_prb1} is essentially the heat equation, and
\eqref{eq:N_jump} implies continuity of $\ut$ and of the flux through
$\bdrp\Osett$. This problem has been already studied in detail in
\cite{Andreucci:2012}; actually the boundary condition
\eqref{eq:N_ext_bc} is replaced there with vanishing Dirichlet data,
but this doesn't affect the well posedness of the model and its
asymptotics.  Instead, in the case of Potassium, condition
\eqref{eq:N_jump} implies the discontinuity of the unknown $\ut$
across $\bdrp\Osett$, namely
\begin{equation}
  \label{eq:jump}
 \inttrace{\ut}
 =
 \frac{\bigD}{\smlD}
 \exttrace{\ut}
 \,.
\end{equation}
Thus $\ut|_{\Osett}$ can not be bounded as $\Mper\to 0$.  But even in
the case of Potassium we can prove the existence of a unique weak
solution $\ut$ for problem \eqref{eq:N_prb1}-\eqref{eq:N_initial} in
the space $\solspace_K$ defined by
\begin{equation}
  \label{eq:mathmod_solspace} 
  \sqrt{\Diffu}\ut
  \in 
  C
  \left
  (0,\maxT,L^2(\Oset)\right)
  \,,
  \quad
  \Diffu\ut
  \in
  L^2
  \left(
  0,\maxT,
  {H}^1(\Oset)
  \right)
  \,,
  \quad
  \ut|_{\openpores}=0
  \,.
\end{equation}
The proof is standard, see however \cite{Andreucci:2012, thesis}.  For
this solution we prove the following, less standard, maximum
principle.
\begin{proposition}
  \label{p:Pmax}
  The solution $\ut$ to \eqref{eq:N_prb1}-\eqref{eq:N_initial} in the
  case of Potassium satisfies
  \begin{alignat}2
    \label{eq:N_max_result}
    0&\le
    {\Diffu(x)\ut(x,t)}
    \le 
    \bigD \Norma{\unk_0}{\oo}
    \,,
    &\qquad &
    (x,t)\in\Oset\times[0,\maxT]
    \,,
    \\
    \label{eq:N_max_result_b}
    0&\le
    {\ut(x,t)}
    \le 
    \Norma{\unk_0}{\oo}
    \,,
    &\quad &
    (x,t)\in(\Oset\setminus \Osett)\times[0,\maxT]
    \,.
  \end{alignat}
\end{proposition}

\begin{proof}
  Let us consider the function $\maxfun(x,t)=\Diffu(x) \ut(x,t)$.
  Let $\maxconst>0$ be a constant to be chosen.
  Using $(\maxfun-\maxconst)_{+}$ as a test function in the weak
  formulation of the problem solved by $\maxfun$ we get
  \begin{equation}
    \label{eq:N_max_princ}
    0
    =
    \frac{1}{\smlD}
    \int_0^{\tfix}
    \int_{\Osett}
    \maxfun_t
    (\maxfun-\maxconst)_{+}
    \di x
    \di t
    +
    \frac{1}{\bigD}
    \int_0^{\tfix}
    \int_{\Oset\setminus \Osett}
    \maxfun_t
    (\maxfun-\maxconst)_{+}
    \di x
    \di t
    -
    \int_0^{\tfix}
    \int_{\Oset}
    \Lapl\maxfun
    (\maxfun-\maxconst)_{+}
    \di x
    \di t
    \,,
  \end{equation}
  and then, setting $\maxconst=\bigD \Norma{\unk_0}{\oo}$ 
  \begin{multline}
    \label{eq:N_max_princ_ii}
    \sup_{0<\tfix<\maxT}
    \frac{1}{\smlD}
    \int_{\Osett}
    (\maxfun-\maxconst)_{+}^2
    (x,\tfix)
    \di x
    +
    \sup_{0<\tfix<\maxT}
    \frac{1}{\bigD}
    \int_{\Oset\setminus \Osett}
    (\maxfun-\maxconst)_{+}^2
    (x,\tfix)
    \di x
    \le
    \\
    \frac{1}{\smlD}
    \int_{\Osett}
    (\maxfun-\maxconst)_{+}^2
    (x,0)
    \di x
    +
    \frac{1}{\bigD}
    \int_{\Oset\setminus\Osett}
    (\maxfun-\maxconst)_{+}^2
    (x,0)
    \di x
    =
    0
    \,.
  \end{multline}
  Hence \eqref{eq:N_max_result} follows, and \eqref{eq:N_max_result_b}
  is an immediate consequence of \eqref{eq:N_max_result}.
\end{proof}

\subsection{Fast and small pores}
\label{ss:main_3d_fs}

The limiting behaviour of the problem
\eqref{eq:N_prb1}--\eqref{eq:N_initial} as $\Mper\to0$ depends sharply
on the relative sizes of the quantities introduced above.
Specifically we need define two possible cases.

Considering the model for Potassium, the cases of \textsl{fast pores}
and \textsl{small pores} are defined respectively by the assumptions
\eqref{eq:main_3d_fast} and \eqref{eq:main_3d_small}.
We also need assume that as $\Mper\to0$
\begin{alignat}2
  \label{eq:main_L2_grd_cfst}
  \fstlK
  \defeq
  \frac{\sqrt{\Mopen}}{\sqrt{\smlD}\Mper}
  \frac{\Mpwid^{\SpDim-1}}{\Mwid^{\SpDim-1}}
  &\to
  \fstconstK
  \,,
  &\qquad&
  \text{in the case of fast pores},
  \\
  \label{eq:main_L2_grd_csml}
  \smllK
  \defeq
  \frac{\Mopen}{\Mper}
  \frac{\Mpwid^{\SpDim-2}}{\Mwid^{\SpDim-1}}
  &\to
  \smlconstK
  \,,
  &\qquad&
  \text{in the case of small pores},
\end{alignat}
where $\fstconstK$, $\smlconstK\in(0,+\oo)$ unless otherwise noted.

Considering Sodium, the cases of \textsl{fast pores} and \textsl{small
  pores} are defined again by \eqref{eq:main_3d_fast} and
\eqref{eq:main_3d_small} \textem{where we formally let} $\smlD=1$.
In each case, we need assume that as $\Mper\to0$
\begin{alignat}2
  \label{eq:main_L2_grd_cfst_Na}
  \fstlNa
  \defeq
  \frac{\sqrt{\Mopen}}{\Mper}
  \frac{\Mpwid^{\SpDim-1}}{\Mwid^{\SpDim-1}}
  &\to
  \fstconstNa
  \,,
  &\qquad&
  \text{in the case of fast pores},
  \\
  \label{eq:main_L2_grd_csml_Na}
  \smllNa
  \defeq
  \frac{\Mopen}{\Mper}
  \frac{\Mpwid^{\SpDim-2}}{\Mwid^{\SpDim-1}}
  &\to
  \smlconstNa
  \,,
  &\qquad&
  \text{in the case of small pores},
\end{alignat}
where $\fstconstNa$, $\smlconstNa\in(0,+\oo)$ unless otherwise noted.

\subsection{Formulation of the limiting problem}
\label{ss:main_3d_limit}

In this Subsection, we always assume that for a constant $\constcp>0$,
\begin{equation}
  \label{eq:limit_hyp}
  \Mwid
  =
  \constcp\sqrt{\smlD\Mper}
  \,,
  \quad
  \Mper>0
  \,;
  \qquad
  \smlD(\tau)\Mper^{-1}
  \to
  +\infty
  \,,
  \quad
  \Mper\to 0
  \,.
\end{equation}
Under the stipulations above, the solution to
\eqref{eq:N_prb1}--\eqref{eq:N_initial} is proven to approximate
as $\Mper\to 0$ the solution to
\begin{alignat}2
 \label{eq:fin_limprb_pde1}
  \unk_{t}
  -
  \Lapl
  (\bigD
  \unk)
  &=
  0
  \,,
  &\qquad&
  \text{in $\Oset\times (0,\maxT)$,}
  \\
  \label{eq:fin_limprb_bdr}
  \grad(\bigD\unk)\cdot \outnormal
  &=
  0
  \,,
  &\qquad&
  \text{on $\bdrp{\Oset}\times (0,\maxT)$}
  \,,
  \\
\label{eq:fin_limprb_inter_i_}
  \grad(\bigD\unk)\cdot \outnormal
  &=
  -
  \limconst \mesfun(x)\bigD
  \unk
  \,,
  &\qquad&
  \text{on $\inter\times (0,\maxT)$}
  \,,
  \\
  \label{eq:fin_limprb_init}
  \unk
  (x,0)
  &=
  \unk_{0}
  (x)
  \,,
  &\qquad&
  \text{in $\Oset$}
  \,,
\end{alignat}
where the constant $\limconst$ is defined in
Theorem~\ref{t:main_3d_conv} below. 

\begin{theorem}
\label{t:main_3d_conv}
As $\tau\to0$ the solution to \eqref{eq:N_prb1}--\eqref{eq:N_initial}
converges to the solution of
\eqref{eq:fin_limprb_pde1}--\eqref{eq:fin_limprb_init}, in the sense
of $\Lsp{2}{\Oset\times(0,\maxT)}$, provided one among
\eqref{eq:main_3d_fast}--\eqref{eq:main_3d_small} is in force, and the
corresponding assumption \eqref{eq:main_L2_grd_cfst} or
\eqref{eq:main_L2_grd_csml} also holds true.

The constant appearing in \eqref{eq:fin_limprb_inter_i_} is defined
by
\begin{equation}
 \label{eq:main_cont_lim_const}
\limconst
=
\begin{cases}
 \frac{2}{\sqrt{\pi}}
\meas{\basepore}_{\SpDim-1}
 \fstconstK
\,,
\qquad
&
\text{in the case of fast pores}
\,,
\\
 \Lflxs
 \smlconstK
\,,
\qquad
&
\text{in the case of small pores,}
\end{cases}
\end{equation}
where $\Lflxs$ is a positive constant related to the geometry
of the pore $\basepore$.
\end{theorem}

The value of $\Lflxs$ can be found in \cite{Friedman:1995,
  Andreucci:2012}.  When the limits in \eqref{eq:main_L2_grd_cfst},
\eqref{eq:main_L2_grd_csml} are either zero or infinite, the limiting
boundary condition 
corresponds to the case of minimal (null) or maximal flux
respectively:

\begin{theorem}
\label{t:neu_bdr}
Assume that \eqref{eq:main_3d_fast}, respectively
\eqref{eq:main_3d_small} is in force.  Then as $\tau\to0$ the solution
to \eqref{eq:N_prb1}--\eqref{eq:N_initial} converges in the sense of
$\Lsp{2}{\Oset\times(0,\maxT)}$ to a function $\unk$ satisfying
\eqref{eq:fin_limprb_pde1}, \eqref{eq:fin_limprb_bdr},
\eqref{eq:fin_limprb_init} and
\begin{alignat}{2}
  \label{eq:fin_limprb_inter_i_neu}
  \grad(\bigD\unk)\cdot \outnormal
  &=
  0
  \,,
  \quad
  \text{on $\inter\times (0,\maxT)$,}
  &\qquad&
  \text{provided $\fstconstK=0$, resp.\ $\smlconstK=0$;}
  \\
 \label{eq:fin_limprb_inter_i_dir}
\mesfun(x)
\unk
  &=
0
\,,
\quad
  \text{on $\inter\times (0,\maxT)$,}
  &\qquad&
  \text{provided $\fstconstK=+\oo$, resp.\ $\smlconstK=+\oo$.}
\end{alignat}
\end{theorem}

\subsubsection{Sketch of the proof of Theorems \ref{t:main_3d_conv},
  \ref{t:neu_bdr}}
\label{s:sketch}

Assuming essentially the same hypoteses of Theorems
\ref{t:main_3d_conv} and \ref{t:neu_bdr}, in \cite{Andreucci:2012} it
has been proven that in the case of Sodium problem
\eqref{eq:N_prb1}--\eqref{eq:N_initial} has a unique solution
converging in the sense of $L^2$ to the solution of
\eqref{eq:fin_limprb_pde1}--\eqref{eq:fin_limprb_init}, with
\begin{equation}
 \label{eq:main_cont_lim_const_Na}
\limconst
=
\begin{cases}
 \frac{2}{\sqrt{\pi}}
\meas{\basepore}_{\SpDim-1}
 \fstconstNa
 \frac{1}{\sqrt{\constD}}
\,,
\qquad
&
\text{in the case of fast pores}
\,,
\\
 \Lflxs
 \smlconstNa
\,,
\qquad
&
\text{in the case of small pores,}
\end{cases}
\end{equation}
where $\Lflxs$ is the same constant as in
\eqref{eq:main_cont_lim_const}.

Since the solution $\ut$ of problem 
\eqref{eq:N_prb1}--\eqref{eq:N_initial} 
satisfies energy estimates which are uniform with respect
to $\Mper$,
by extracting a subsequence if needed, and also using the bounds of
Proposition~\ref{p:Pmax}, we have
\begin{gather*}
  \ut 
  |_{x\in\Oset\setminus \Osett}
  \to\unk
  \,,
  \qquad
  \text{strongly in $L^{2}(\Oset\times(0,\maxT))$
  as $\Mper \to 0$;}
  \\
  \grad(\bigD\ut)
  |_{x\in\Oset\setminus \Osett}
  \to
  \grad(\bigD\unk)
  \,,
  \qquad
  \text{weakly in $L^{2}(\Oset\times(0,\maxT))$
  as $\Mper \to 0$.} 
\end{gather*}
Then it is easy to see that $\unk$ satisfies
\eqref{eq:fin_limprb_pde1}, \eqref{eq:fin_limprb_bdr},
\eqref{eq:fin_limprb_init} in a standard weak sense. This simple
compactness argument leaves completely unsolved the problem of
determining the limiting boundary condition satisfied for $x\in
\inter$. The latter condition may be identified as in
\cite{Andreucci:2012}, relying on a careful analysis of the behaviour
of a suitable oscillating test function defined in $\Osett$. Actually,
since in $\Osett$ the diffusivity $\smlD$ is constant, the analysis of
\cite{Andreucci:2012} carries through essentially without technical
changes.

However the introduction of a vanishing diffusivity $\smlD$
in our problem is important in the function $\fstlK$ approximating
$\fstconstK$, and in the assumptions on $\Mwid$ stipulated in Theorems
\ref{t:main_3d_conv} and \ref{t:neu_bdr}; the consequences of this
are in our opinion interesting and will be discussed below.

\subsubsection{Discussion: Potassium flux enhancement using Fokker-Planck equation}
  \label{ss:discussion}

  In this subsection we give an interpretation of the results stated
  above and show that using as a starting point the model in
  \cite{Andreucci:2012} modified as above by introducing a vanishingly
  small diffusivity in a neighborhood of the pores, we can mimic the
  selectivity mechanism present in many biological membranes. This
  effect also relies on the use of the Fokker-Planck equation, which
  implies the interface condition \eqref{eq:N_jump} and therefore the
  jump relation \eqref{eq:jump}, which is instrumental in the
  enhancement of the local concentration, and therefore of the
  outflux; see also \cite{Sattin:2004}.

  The Propositions below, together with Theorems \ref{t:main_3d_conv},
  \ref{t:neu_bdr}, show that, if Sodium and Potassium share a common
  set of parameters $\Mwid$, $\Mpwid$, $\Mopen$, in the case of fast
  pores for Potassium (that could correspond either to fast or small
  pores for Sodium), if the limiting boundary condition on $\inter$
  for Potassium [Sodium] is \eqref{eq:fin_limprb_inter_i_}, then for
  Sodium [Potassium] it is of Neumann [Dirichlet] type. Therefore, in
  this case we proved an enhanced asymptotic flux for Potassium with
  respect to Sodium.
  \\
  On the other hand in the case of small pores for both Potassium and
  Sodium, if the species share the same set of parameters, they also
  share the limiting behaviour of the boundary condition on $\inter$.
  The case of small pores is therefore not sensitive to the mechanism
  of selection we introduced for Potassium.
 
\begin{proposition}
 \label{pr:compare_fast}
 Assume
 \eqref{eq:main_3d_fast} and \eqref{eq:main_L2_grd_cfst} with
 $0<\fstconstK<+\oo$.  If either case of fast or small pores holds
 true for Sodium, then the corresponding limit relation in
 \eqref{eq:main_L2_grd_cfst_Na} or in \eqref{eq:main_L2_grd_csml_Na}
 is satisfied with $\fstconstNa=0$ or $\smlconstNa=0$.
 \\
 Conversely, assume \eqref{eq:main_3d_fast} but not
 \eqref{eq:main_L2_grd_cfst}. If either case of fast or small pores
 holds true for Sodium, and if the corresponding limit relation in
 \eqref{eq:main_L2_grd_cfst_Na} or in \eqref{eq:main_L2_grd_csml_Na}
 is satisfied with $0<\fstconstNa<+\oo$ or $0<\smlconstNa<+\oo$, then
 then the limit relation \eqref{eq:main_L2_grd_cfst} is satisfied with
 $\fstconstK=+\oo$.
\end{proposition}

\begin{proposition}
 \label{pr:compare_small}
 Assume \eqref{eq:main_3d_small} and \eqref{eq:main_L2_grd_csml}
 with $0<\smlconstK<+\oo$.
 Then the case of small pores for Sodium holds true, and the limit
 relation \eqref{eq:main_L2_grd_csml_Na} is satisfied with
 $\smlconstNa=\smlconstK$.
 \\
 Conversely, assume \eqref{eq:main_3d_small}, but not
 \eqref{eq:main_L2_grd_csml}. If the case of small pores for Sodium
 holds true, as well as
 \eqref{eq:main_L2_grd_csml_Na} with $0\le\smlconstNa \le +\oo$,
 then the limit relation \eqref{eq:main_L2_grd_csml} is satisfied 
 with $\smlconstK=\smlconstNa$. 
\end{proposition}

The proofs of Propositions \ref{pr:compare_fast} and
\ref{pr:compare_small} follow from some simple algebra and the
definitions \eqref{eq:main_3d_fast}--\eqref{eq:main_3d_small},
\eqref{eq:main_L2_grd_cfst}--\eqref{eq:main_L2_grd_csml_Na}.

%% file: onedim/onedim.tex
\subsection{The problem for $N=1$ and $\smlD$ not depending on $\Mper$}
\label{s:N=1}
Having in mind the application of next Sections, we look here at the
approximating problem for Potassium given by the $1$-dimensional
version of \eqref{eq:N_prb1}--\eqref{eq:N_initial}, where however
$0<\smlD<\bigD$ are given constants and we set for $0<\Mdpt<\Mbor$
\begin{equation}
 \label{eq:1d_D_definition}
 \Diffu(x)
 =
 \begin{cases}
  \bigD
  \,,
  \quad&
  x\in [0,\Mbor-\Mdpt]
  \\
  \smlD
  \,,
  \quad&
  x\in (\Mbor-\Mdpt,\Mbor]
 \end{cases}
\,.
\end{equation}
Indeed in the numerical simulations it would be technically very
difficult to mimic the limit $\smlD\to0$.  We remark that the boundary
$x=\Mbor$ still is a pore alternating with period $\Mper$ and open
phase $\Mopen$.  By methods similar to those outlined in
Subsection~\ref{s:sketch} we can prove that the solution $\ut$ to this
problem approximates as $\Mper\to 0$ the solution to
\begin{alignat}2
 \label{eq:1d_fin_limprb_pde1}
  \unk_{t}
  -
  (\bigD
  \unk)_{xx}
  &=
  0
  \,,
  &\qquad&
  \text{in $\Oset\times (0,\maxT)$,}
  \\
  \label{eq:1d_fin_limprb_bdr}
  (\bigD\unk)_x
  &=
  0
  \,,
  &\qquad&
  \text{for $x=0$ and $t\in(0,\maxT)$}
  \,,
  \\
\label{eq:1d_fin_limprb_inter_i_}
  (\bigD\unk)_x
  &=
  -
  \limconst_1\frac{\bigD}{\sqrt{\smlD}}
  \unk
  \,,
  &\qquad&
  \text{for $x=\Mbor$ and $t\in(0,\maxT)$}
  \,,
  \\
  \label{eq:1d_fin_limprb_init}
  \unk
  (x,0)
  &=
  \unk_{0}
  (x)
  \,,
  &\qquad&
  \text{in $\Oset$}
  \,,
\end{alignat}
where the constant $\limconst_1$ is defined in
Theorem~\ref{t:main_1d_conv} below.

\begin{theorem}
\label{t:main_1d_conv}
Let $\Mdpt=\sqrt{\smlD \Mper}$ and assume that
\begin{equation}
 \label{eq:1d_lim_cond}
 \lim_{\Mper\to 0}
 \frac{\sqrt{\Mopen}}{\Mper}
 =
 \constidK
 \,,
\end{equation}
where $\constidK$ is a positive constant.  As $\tau\to0$ the solution
$\ut$ converges to the solution of
\eqref{eq:1d_fin_limprb_pde1}--\eqref{eq:1d_fin_limprb_init}, in the
sense of $\Lsp{2}{\Oset\times(0,\maxT)}$. The constant $\limconst_{1}$
in \eqref{eq:1d_fin_limprb_inter_i_} is defined by
\begin{equation}
 \label{eq:1d_main_cont_lim_const}
\limconst_1
=
 \frac{2}{\sqrt{\pi}}
 \constidK
\,.
\end{equation}
\end{theorem}

When the limit in \eqref{eq:1d_lim_cond} is 
either zero or infinite, the limiting boundary condition is different.

\begin{theorem}
\label{t:1d_neu_bdr}
Assume $\Mdpt=\sqrt{\smlD \Mper}$ and \eqref{eq:1d_lim_cond}.  Then as
$\tau\to0$ the solution $\ut$
converges in the sense of $\Lsp{2}{\Oset\times(0,\maxT)}$ to a
function $\unk$ satisfying \eqref{eq:1d_fin_limprb_pde1},
\eqref{eq:1d_fin_limprb_bdr}, \eqref{eq:1d_fin_limprb_init} and
\begin{alignat}{2}
  \label{eq:1d_fin_limprb_inter_i_neu}
  (\bigD\unk)_x
  &=
  0
  \,,
  \quad
  \text{for $x=\Mbor$ and $t\in(0,\maxT)$,}
  &\qquad&
  \text{provided $\constidK=0$;}
  \\
 \label{eq:1d_fin_limprb_inter_i_dir}
\unk
  &=
0
\,,
\quad
  \text{for $x=\Mbor$ and $t\in(0,\maxT)$,}
  &\qquad&
  \text{provided $\constidK=+\oo$.}
\end{alignat}
\end{theorem}

In the case of Sodium the same results hold, replacing $\smlD$ with
$\bigD$ everywhere above in this Subsection.

If Theorem~\ref{t:main_1d_conv} holds, from
\eqref{eq:1d_fin_limprb_inter_i_} we have the following effect of
enhanced flux by selection: The asymptotic ratio between outgoing flux
and concentration $-{(\bigD \unk)_x}/{\unk}$ at $x=\Mbor$ in the case
of Potassium is bigger with respect to the case of Sodium by a factor
$\sqrt{{\bigD}/{\smlD}}$.

%% file: onedim/onedim_lattice_model.tex
\section{A discrete space model}
\label{s:rw}
Next we approach the problem via a discrete space model.  In this
section we first define the model and then discuss heuristically the
relation between the outgoing flux and the ion density close to the
pore.  In next section this model will be studied via Monte Carlo
simulations.

We consider $M$ one--dimensional independent random walkers on
$H=H_0\cup H_1$ with
$H_0=\{\ell,2\ell,\dots,n_0\ell\}\subset\ell\bb{Z}$ and
$H_1=\{(n_0+1)\ell,(n_0+2)\ell,\dots,(n_0+n_1)\ell\}\subset\ell\bb{Z}$,
where $n_0$ and $n_1$ are non--negative integers.  We denote by $t\in
s\bb{Z}_+$ the time variable.  We assume the following: (i) each
random walk is symmetric, (ii) only jumps between neighboring sites
are allowed, (iii) in the region $H_1$ particles have the probability
$r\in[0,1]$ not to move, (iv) $0$ is a reflecting boundary point, and
(v) picked the two integers $1\le \bar\sigma\le \bar\tau$, we
partition the time space $s\bb{Z}_+$ in
\begin{equation*}
  A=
  \bigcup_{i=1}^\infty
  \{s(i-1)\bar\tau,\dots,s[(i-1)\bar\tau+\bar\sigma-1]\}
\;\;\;\textrm{ and }\;\;\;
  C=
  \bigcup_{i=1}^\infty
  \{s[(i-1)\bar\tau+\bar\sigma],\dots,s[i\bar\tau-1]\}
\end{equation*}
and assume that the boundary point $(n_0+n_1+1)\ell$ is absorbing at times 
in $A$ (open phase) and reflecting at times in $C$ (closed phase). 

To be more precise, we write explicitly the probability 
$p(x,y)$ that the walker at site $x$ jumps to site $y$. We first set 
$r(x)=0$ if $x\in H_0$ and $r(x)=r$ if $x\in H_1$, then we have
\begin{equation*} 
p(\ell,\ell)=\frac{1}{2},\;\;\; 
p(x,x+\ell)=\frac{1-r(x)}{2}\;\textrm{ for }\;x=\ell,\dots,(n_0+n_1-1)\ell,
\end{equation*} 
and
\begin{equation*} 
p(x,x-\ell)=\frac{1-r(x)}{2}\;\textrm{ for }\;x=2\ell,\dots,(n_0+n_1)\ell\,.
\end{equation*} 
Moreover 
\begin{displaymath}
 p((n_0+n_1)\ell,(n_0+n_1)\ell)= 
 \left\{
 \begin{array}{ll}
 r(x) & \textrm{at times in } A\\
 (1+r(x))/2 & \textrm{at times in } C\\
 \end{array}
 \right.
\end{displaymath}
and
\begin{displaymath}
 p((n_0+n_1)\ell,(n_0+n_1+1)\ell)= 
 \left\{
 \begin{array}{ll}
 (1-r(x))/2 & \textrm{at times in } A\\
 0 & \textrm{at times in } C\,.\\
 \end{array}
 \right.
\end{displaymath}

Notice that when the walker reaches the site $(n_0+n_1+1)\ell$ it is 
frozen there, so that this system is a model for the proposed 
problem in the following sense: each walker is an ion, 
the cell is the set $ H_0\cup H_1$, 
the low diffusivity region close to the pore is the set $ H_1$
(indeed, there the particles move less
frequently and, hence, diffuse at a slower rate),
at the initial time there are $M$ ions in the cell, 
each ion absorbed at the site $(n_0+n_1+1)\ell$ is counted as 
an ion which exited the cell. 
It is also important to note that the case $n_1\ge1$ models the 
Potassium problem, 
whereas the case $n_1=0$ 
models the Sodium problem. 
The Sodium--like case  has been dealt upon in 
\cite{Andreucci:2013b}, hence from now on we assume $n_1\ge1$. 

It is important to note that for $\bar\sigma=0$, namely, when 
the pore is always closed, each walker admits the following unique
stationary measure: 
the probability that a site in the region 
$ H_0$ is occupied by the walker is equal to 
$(1-r)/[(1-r)n_0+n_1]$, 
the probability that a site in the region 
$ H_1$ is occupied by the walker is equal to 
$1/[(1-r)n_0+n_1]$. This state will be called in the sequel 
the \textit{closed pore stationary state}. In this state, the typical 
number of particle on a site in $ H_0$ (resp.\ in $ H_1$) 
is given by 
$M(1-r)/[(1-r)n_0+n_1]$
(resp.\ $M/[(1-r)n_0+n_1]$). 
We denote by $\bb{P}[\cdot]$ and 
$\bb{E}[\cdot]$ the probability and the average along the trajectories of the 
process started at the closed pore stationary state.

When the pore is opened for the first time, 
the initial state is perturbed as an effect 
of the outgoing flux of particles; at the end of the first opening cycle 
the total number of particles in the system will be smaller than $M$. 
When the pore is closed, the system 
tends to restore the closed pore stationary state with the new value 
of the total particle number. We will always assume that 
\begin{equation}
\label{assumo01}
\bar\tau\gg\bar\sigma
\;\;\;\textrm{ and }\;\;\;
\bar\tau> n_1^2
\end{equation}
so that we can reasonably think that at the beginning of each 
opening cycle the distribution of particles throughout the region 
$ H_1$ 
is approximatively constant.
Indeed, under 
this hypothesis the time interval in which the 
right hand boundary point is absorbing is much smaller than
that in which it is reflecting; in other words in each cycle
the pore is open in a very short time subinterval.

In the framework of this model an estimator for the 
ratio between the outgoing ion flux and 
the typical number of particles in the high diffusivity region 
but close to the low diffusivity one
is given by 
\begin{equation}
\label{estimatore}
K_i=
\frac{\bb{E}[F_i]/(s\bar\tau)}{(\bb{E}[U_i]/\bar{\tau})/\ell}
\;\;\;\;\;\;\textrm{ for all }i\in\bb{Z}_+
\end{equation}
where
$F_i$ is the number of walkers that reach the boundary point 
$(n_0+n_1+1)\ell$ 
during the $i$--th cycle,
$U_i$ is the sum over the time steps in the $i$--th cycle 
of the number of walkers at the site $n_0\ell$.

We are interested in two main problems. 
The first question that we address is the dependence 
on time of the above ratio, in other words we wonder if this quantity 
does depend on $i$.
The second problem that we investigate is the connection between the 
predictions of this discrete time model and those provided by 
the continuous space one introduced in 
Subsection~\ref{s:N=1}.

\subsection{The estimator $K_i$ is a constant}
\label{s:time}
\par\noindent
As remarked above, under 
the assumption \eqref{assumo01}, it is reasonable 
to guess that during any cycle the walkers in the region 
$ H_1$ 
are distributed uniformly with a very good approximation. 
Hence, at each time and 
at each site of $ H_1$ the number of walkers on that site 
is approximatively given a constant denoted by $v_i$.
Since $\bar\sigma$ is much smaller than $\bar\tau$, the mean number of 
walkers $\bb{E}[F_i]$ that reach the boundary point $(n_0+n_1+1)\ell$ 
during the 
cycle $i$ is proportional to $v_{i-1}$ and the 
constant depends on $\bar\sigma$, so that we have 
\begin{equation}
\label{alpha}
\bb{E}[F_i]=\alpha(\bar\sigma)v_{i-1}
\,.
\end{equation}
We also note that, since $\bar\tau\gg\bar\sigma$, we have that 
\begin{displaymath}
n_1v_i
=
n_1v_{i-1}
-\bb{E}[F_i]
+\Delta_i
\end{displaymath}
where $\Delta_i$ is the expected
difference between the number of particles 
that during the cycle $i$ moved from the region $ H_0$ to the 
region $ H_1$ and that of the particle that moved in the 
opposite direction; note that $\Delta_i$ admits the obvious bound 
$\Delta_i\le M$. 
At the end of each cycle we can assume that a 
sort of stationarity is achieved on the boundary 
between $ H_0$ and $ H_1$; so that we can assume 
$\bb{E}[U_i]/(2\bar\tau)=v_i (1-r)/2$. 
By using this remark, the two equations above, and the fact that 
$\Delta_i\le M$, 
we get that 
\begin{equation}
\label{estimatore02}
K_i
\stackrel{n_1\to\infty}{\approx}
K
\equiv
\alpha(\bar\sigma)
\frac{1}{\bar\tau}
\frac{\ell}{s}
\frac{1}{1-r}
\end{equation}
showing that, provided $n_1$ is large enough,  
the estimator \eqref{estimatore} 
does not depend on time, namely, it is approximatively 
equal to $K$ for each $i$. 

\subsection{Behavior of the constant $\alpha$ for large $\bar\sigma$}
\label{s:alpha}
\par\noindent
We are, now, interested in finding an estimate for 
$\alpha(\bar\sigma)$ in the limit when $\bar\sigma$ is large. 
The reason 
why we need this kind of result will be discussed in 
the following section.

If $\bar\sigma$ is large, 
at time $\bar\sigma$ each walker space distribution probability 
can be approximated by a Gaussian function with variance 
$\sqrt{2 \bar\sigma(1-r)}$ (Central Limit Theorem). 
Hence, the number of particles that reach in 
$\bar\sigma$ steps the boundary $(n_0+n_1+1)\ell$ is 
approximatively given by the number of walkers at the 
$\sqrt{2 \bar\sigma(1-r)}$ sites counted starting from the absorbing 
boundary point divided by $2$. Hence, we find the rough estimate 
\begin{equation}
\label{restim}
\alpha(\bar\sigma)
\approx
\frac{1}{2}\sqrt{2\bar\sigma(1-r)}
=
\sqrt{\frac{\bar\sigma(1-r)}{2}}
\end{equation}
suggesting that, for large $\bar\sigma$, 
the quantity $\alpha(\bar\sigma)$ depends on $\bar\sigma$
as $\sqrt{\bar\sigma}$.

\subsection{Comparison with the continuum space model}
\label{s:continuo}
\par\noindent
In order to compare the results discussed above in this section with those in 
Subsection~\ref{s:N=1} referring to the continuous space 
model defined therein, we have to consider two limits. 
The parameter $\bar\sigma$ has to be taken large (recall, 
also, that we always assume $\bar\tau\gg\bar\sigma$,
see \eqref{assumo01})
so that, 
due to the Central Limit Theorem,  
the discrete and the continuous space model have similar behaviors 
provided the other parameters are related properly.
With a correct choice of the parameters, then, 
we expect that the discrete space model 
will give results similar to those predicted by the 
continuous space one. 
In Subsection~\ref{s:N=1}, see Theorem~\ref{t:main_1d_conv}, 
the relation between 
the outgoing flux and the density close to the pore is worked out only 
in the limit $\tau\to0$. We then have to understand how to implement such 
a limit in our discrete time model. 

We perform this analysis in the critical case $\sigma_\tau=\mu^2\tau^2$;
note the the hypothesis in Theorem~\ref{t:main_1d_conv} is weaker, 
indeed, there $\ell_{1K}$ is the limit for $\tau\to0$ of the 
ratio $\sqrt{\sigma_\tau}/\tau$, see equation \eqref{eq:1d_lim_cond}.
From now on we let 
$L_0=a-\delta$ and $L_1=\delta$, see Subsection~\ref{s:N=1}.
We imagine to fix the continuous model parameters and then 
choose properly the discrete space model ones. 
More precisely we assume given $L_0$, $D_0$, $D_1$, and $\mu$, and 
recall that $L_1$ is related to the other parameters by the equation
\begin{equation}
\label{l1tau}
L_1=\sqrt{D_1\tau}
\;\;\; .
\end{equation}
However, notice that, as an immaterial technical change, 
here $L_0$ is fixed rather than $L_0+L_1$ as in Subsection~\ref{s:N=1}.

We now describe our procedure in detail:
in order to compare the discrete and the continuum space models we first let 
\begin{equation}
\label{lenght}
L_0=\ell n_0
\;\;\;\textrm{ and }\;\;\;
L_1=\ell n_1
\,.
\end{equation}
These two equations 
yield an expression for $\ell$ and the relation that 
must be verified by $n_0$ and $n_1$, more precisely we get
\begin{equation}
\label{length02}
\ell=\frac{L_0}{n_0}
\;\;\;\textrm{ and }
n_1=n_0\,\frac{L_1}{L_0}
\,.
\end{equation}

As already remarked, 
from the Central Limit Theorem, it follows 
that the two models give the same long time predictions if 
$2D_0s=\ell^2$ and 
$2D_1s=\ell^2(1-r)$. We then get an expression for the unit time 
and a relation between $D_1$ and $r$, namely, 
\begin{equation}
\label{time}
s
=\frac{\ell^2}{2D_0}
=\frac{L_0^2}{2D_0n_0^2}
\;\;\;\textrm{ and }\;\;\;
D_1=D_0(1-r)
\,.
\end{equation}
We then consider the random walk model introduced above 
by choosing $\bar\sigma$ and $\bar\tau$ such that 
the equality 
$\bar\sigma s = (\mu \bar\tau s )^2$ 
is satisfied as closely as possible (note that $\bar\tau$ and $\bar\sigma$ 
are integers). 
This can be done as follows: recall that $L_0$, and $\mu$ are fixed; 
we choose also $n_0$ and $\bar\sigma$, and set  
\begin{equation}
\label{ntau}
\bar\tau
=
\bigg\lfloor
\frac{1}{\mu}
\sqrt{\frac{\bar\sigma}{s}}
\bigg\rfloor
=
\frac{1}{\mu}
\frac{n_0}{L_0}
\sqrt{2D_0\bar\sigma}
-b
\end{equation}
where $\lfloor\cdot\rfloor$ denotes the integer part of a real number and 
$b\in[0,1)$. 
With the above choice of the parameters, the behavior of the 
random walk model has to be compared with that of the continuum space 
model in Subsection~\ref{s:N=1} with period
\begin{equation}
\label{tau}
\tau
=
s\bar\tau
=
\frac{1}{\mu}
\frac{L_0\sqrt{\bar\sigma}}{n_0\sqrt{2D_0}}
-
\frac{L_0^2}{2D_0n_0^2}
b
\,.
\end{equation}

The equation \eqref{tau} is very important in our computation, since it 
suggests that the homogenization limit $\tau\to0$ studied in the 
continuum model should be captured by the discrete space model 
via the thermodynamics limit $n_0\to\infty$. 
We also note that, from the equations above, one gets
\begin{displaymath}
n_1
=
n_0\frac{L_1}{L_0}
=
\frac{n_0}{L_0}\sqrt{D_1\tau}
=
\frac{n_0}{L_0}\sqrt{D_0(1-r)s\bar\tau}
=
\frac{n_0}{L_0}\sqrt{D_0(1-r)}
\frac{L_0}{n_0}\frac{1}{\sqrt{2D_0}}
\sqrt{\bar\tau}
=
\sqrt{\frac{1-r}{2}}
\sqrt{\bar\tau}
\end{displaymath}
showing that in the thermodynamics limit $n_0\to\infty$ also the parameter 
$n_1$ tends to infinity.
In conclusion, from the continuous space model, we 
expect that the estimator $K$ converges to the constant 
$2\mu D_0/\sqrt{\pi D_1}$ in this limit. 
In the next section we shall check this result 
via a Monte Carlo computation, but here we argue this guess
has a chance to be correct on the basis of the 
rough estimate 
(\ref{restim}).

Indeed, we first note that  
by \eqref{estimatore02}, \eqref{ntau}, the first of equations 
\eqref{length02},
and the first of equations 
\eqref{time},
we have that
\begin{equation}
\label{estimatore03}
K
=
\frac{\alpha(\bar\sigma)}{\sqrt{\bar\sigma}}
\frac{1}{1-r}
\mu
\sqrt{2D_0}
\,.
\end{equation}
Then, by (\ref{restim}), we get that
\begin{displaymath}
K
\stackrel{\bar\sigma\to\infty}{\longrightarrow}
\frac{1}{2}\sqrt{2(1-r)}
\,
\frac{1}{1-r}
\,
\mu
\,
\sqrt{2D_0}
=
\frac{\mu}{\sqrt{1-r}}
\,
\sqrt{D_0}
=
\frac{\mu D_0}{\sqrt{D_1}}
\end{displaymath}
where in the last step we have used the second among the equations 
\eqref{time}.
Note that this heuristic result is very close to the 
desired limit, at least the dependence on the diffusion coefficient 
is correct. The prefactor is wrong due to the poor estimate \eqref{restim}
that
we have for the constant $\alpha(\bar\sigma)$.

%% file: onedim/onedim_mc_results.tex
\section{Monte Carlo results}
\label{s:discu}
\par\noindent
In this section we describe the Monte Carlo computation of the 
constant \eqref{estimatore}. This measure is quite difficult since 
in this problem the stationary state is trivial, in the sense 
that, since there is an outgoing flux through the boundary point 
$(n_0+n_1+1)\ell$ 
and no ingoing flux is present, all the particles will eventually exit 
the system itself. 
On the other hand, the measure that we have to perform 
is intrinsically non--stationary.
Indeed, 
our problem can be stated as follows: both the outgoing 
flux and the local density at the boundary between the low and the 
high diffusivity regions
are two in average decreasing 
random variables, but their mutual ratio is constant in average. 
We then have to set up a procedure to capture this constant ratio. 

\begin{figure}[!h]
  \centering
   {\epsfig{file = 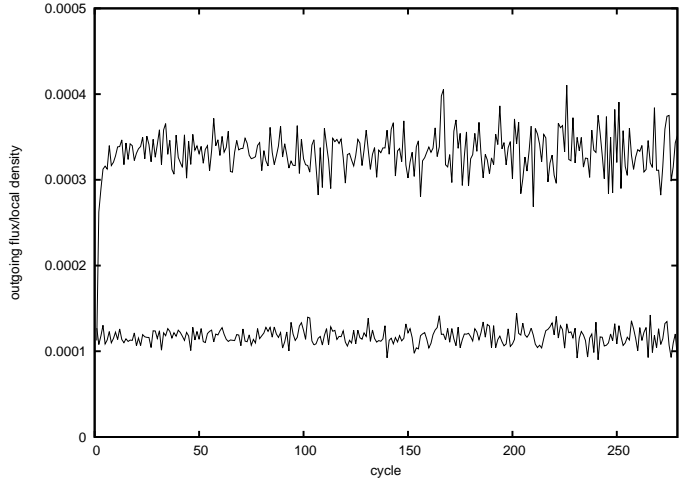, width = 6.5cm, angle=-90}}
  \caption{The quantity $k_i$ is plotted vs.\ the cycle number $i$ in the 
           case $D_1=0.1$, $\bar\sigma=1000$, and $n_0=5023$.
           The top curve refers to the Potassium case, while the bottom 
           one refers to the Sodium one ($n_1=0$). 
          }
  \label{f:kiuno}
  \vspace{-0.1cm}
\end{figure}

We fix the parameters 
$L_0=1$, 
$D_0=1$, 
and
$\mu=1$.
For the diffusion close to the pore we shall consider two 
cases, $D_1=0.1$ and $D_1=0.25$; 
the corresponding values for the parameter $r$, see the second 
of equations \eqref{time}, 
are $0.9$ and $0.75$.
We note that in these two cases the 
continuous space model limit for the constant $K$ is 
respectively given by 
$2\mu D_0/\sqrt{\pi D_1}=3.568$
and 
$2\mu D_0/\sqrt{\pi D_1}=2.257$. 
For the time length of the open state,
we shall consider the following values 
\begin{displaymath}
\bar{\sigma}
=
500\,,
1000\,,
2000\,,
5000
\,.
\end{displaymath}
For each of them, in order to 
perform the limit $\tau\to0$, we shall consider different values of $n_0$, 
ranging from about
$1000$ to $20000$, for the number of sites of the lattice $ H_0$. All
the other parameters will be computed via the equations 
discussed in Subsection~\ref{s:continuo}. As initial number of particles 
we used $M=10^5$. 

For each choice of the two parameters $\bar{\sigma}$ and $n_0$ we shall 
run the process and compute at each cycle $i$ the quantity 
\begin{displaymath}
k_i
=
\frac{F_i/(\bar\tau)}{U_i/(\bar\tau)}
\end{displaymath}
where, we recall, $\bar\tau$ is defined in \eqref{ntau} and 
$F_i$ and $U_i$ have been defined below \eqref{estimatore}.

The quantity $k_i$ is a random variable 
fluctuating with $i$, but, as it is illustrated in 
Figure~\ref{f:kiuno}, it performs random 
oscillations around a constant reference value. 
We shall measure this reference value by computing 
the time average of the quantity $k_i$. We shall average $k_i$ by 
neglecting the initial cycles 
and the very last one which are characterized by large 
oscillations due to the smallness of the number of residual particles in 
the system. 

The product of the reference value for the random variable $k_i$ 
and the quantity $\ell/s$, see the equations \eqref{estimatore}, 
\eqref{length02} and \eqref{time},
will be taken as 
an estimate for $K$.  In other words 
the output of our computation will be the quantity 
\begin{equation}
\label{costK}
K=
\frac{\ell}{s}\times (k_i\textrm{ time average}) 
\,.
\end{equation}

\begin{figure}[!h]
  \centering
   {\epsfig{file = 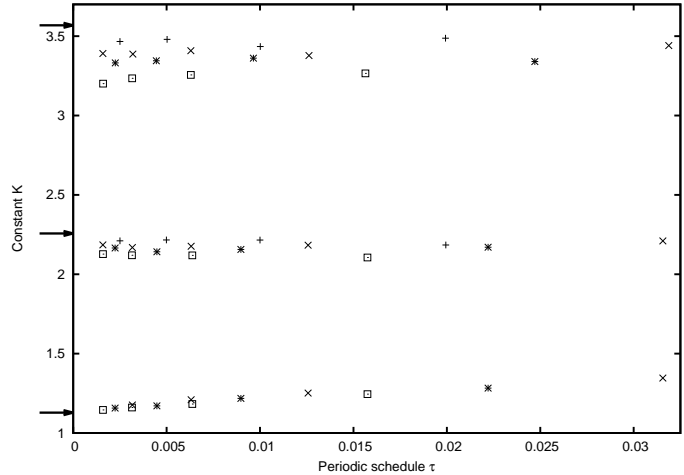, width = 6.5cm, angle=-90}}
  \caption{Monte Carlo estimate of the constant $K$ measured as in  
           \eqref{costK} vs.\ the periodic time schedule $\tau$ for 
           the Potassium model with 
           $D_1=0.1$ at the top and
           $D_1=0.25$ in the middle, and for the Sodium case at the bottom.
           The symbols $+$, $\times$, $*$, and 
           $\boxdot$ refer, respectively, to the cases 
           $\bar\sigma=5000,2000,1000,500$.
           The three arrows, from the top to the bottom, 
           indicate the three theoretical limits 
           3.568, 2.257, and 1.1284
           corresponding respectively to the three cases  
           Potassium $D_1=0.1$, Potassium $D_1=0.25$, and Sodium.
           Notice that the Monte Carlo $\tau\to0$ limit approximate 
           the theoretical one better and better when $\bar\sigma$ 
           is increased.
}
  \label{f:costante}
  \vspace{-0.1cm}
\end{figure}

Our numerical results are illustrated in Figure~\ref{f:costante}.
In \cite{Andreucci:2013b}, for the Sodium case, we noted that 
by increasing $\bar\sigma$ the numerical series tended to collapse 
to one limiting behavior. In that paper we discussed Monte 
Carlo results in the cases $\bar\sigma=30, 50, 70, 100, 120, 150, 200$. 
In this paper we consider larger values of $\bar\sigma$ and, as we 
expected, the numerical series for the Sodium case collapse 
to one single curve whose $\tau\to0$ limit is very close 
to the theoretical value $2\mu\sqrt{D_0}/\sqrt{\pi}=1.1284$, 
see \cite[Section~3.3]{Andreucci:2013b}.

The numerical study is more complicated in the Potassium case, 
since after each opening cycle the system tends to restore 
a new closed pore stationary state with two different 
typical densities in the regions $ H_0$ and $ H_1$. 
It is not really possible to estimate how efficient is 
this restoring process.
However, our numerical estimates 
are perfectly in agreement with the theoretical prediction. 

Again, we note that when $\bar\sigma$ is increased the 
numerical series tend to be mutually closer and closer, even if we cannot 
observe a precise collapse. But it looks reasonable to suppose 
that, if larger values of $\bar\sigma$ were considered, a 
complete collapse could be obtained. Considering larger value 
of $\bar\sigma$ would be extremely time consuming from 
the point of view of numerical simulations, indeed larger and 
larger values of the lattice size $n_0$ would have to be used.

The limiting behavior for $\tau\to0$, that in this numerical
scheme is achieved via a thermodynamics limit $n_0\to\infty$, 
reproduces quite well the theoretical prediction based on the 
homogenization computation discussed in Subsection~\ref{s:N=1}. 
Indeed, the data in Figure~\ref{f:costante} show neatly that the 
series with the largest $\bar\sigma$ approach, for $\tau\to0$, a value 
quite close to the theoretical predictions 
$3.568$ (case $D_1=0.1$)
and 
$2.257$ (case $D_1=0.25$).

We can finally state that the Monte Carlo measure of the constant $K$
is in very good agreement with the theoretical predictions discussed above
which, we recall, are based on a homogenization computation in 
presence of a spatial discontinuity of the diffusion coefficient.